\documentclass[11pt]{amsart}
\usepackage{amsfonts}
\usepackage{amsmath}
\usepackage{amssymb}
\usepackage{graphicx}
\usepackage{amscd}
\usepackage{thmdefs}
\usepackage[colorlinks,urlcolor=blue,citecolor=blue,filecolor=blue,linkcolor=blue]{hyperref}

\input{tcilatex}
\numberwithin{equation}{section}

\input tcilatex

\newtheorem{st}{Statement}[section]

\newtheorem{propositionA}[st]{Proposition}
\newtheorem{exampleA}[st]{Example}
\newtheorem{corollaryA}[st]{Corollary}

\newtheorem{theoremA}[st]{Theorem}
\newtheorem{lemmaA}[st]{Lemma}

\begin{document}
\title[Continuous extensions]{Continuous extension of maps between
sequential cascades}
\author{Szymon Dolecki}
\date{%
\today
}
\address{Institut de Math\'{e}matiques de Bourgogne, B. P. 47870, 21078
Dijon, France.}
\email{dolecki@u-bourgogne.fr}
\author{Andrzej Starosolski$^{0}$}
\address{Instytut Matematyki, Politechnika \'{S}l\c{a}ska, ul. Kaszubska 23\\
44-100 Gliwice, Poland.}
\email{Andrzej.Starosolski@polsl.pl}
\keywords{sequential cascade, contour, extension of maps \newline \hspace*{3mm} 2010
Mathematics Subject Classification 54C20, 03E05, 54D55}
\maketitle

\begin{abstract}
The contour of a family of filters along a filter is a set-theoretic lower
limit. Topologicity and regularity of convergences can be characterized with
the aid of the contour operation. Contour inversion is studied, in
particular, for iterated contours of sequential cascades. A related problem
of continuous extension of maps between maximal elements of sequential
cascades to full subcascades is solved in full generality.
\end{abstract}

\section{Introduction\label{intr}}

\footnotetext{%
The corresponding author.}

The \emph{contour} $\limfunc{Li}\nolimits_{\left( n\right) }\mathcal{F}_{n}$
of a sequence of filters $\left( \mathcal{F}_{n}\right) _{n}$ on $X$,
defined by%
\begin{equation*}
\limfunc{Li}\nolimits_{\left( n\right) }\mathcal{F}_{n}:=\bigcup%
\nolimits_{k<\omega }\bigcap\nolimits_{n>k}\mathcal{F}_{n},
\end{equation*}%
is a special case of the \emph{contour operation}. This operation is
important because it can be used, in complementary ways, to define both
diagonal and regular convergences, hence to characterize topologicity and
regularity. It was also used by Frol\'{\i}k in \cite{Frolik} to show in ZFC
the non-homogeneity of the compactification of a countably infinite set (See
Section \ref{app} for more information).

In this paper, we investigate how much information on $\left( \mathcal{F}%
_{n}\right) _{n}$ can be recovered from the knowledge of $\limfunc{Li}%
\nolimits_{\left( n\right) }\mathcal{F}_{n},$ that is, a contour inversion
problem. In some situations, the contour determines the original sequence
almost entirely. For example, it follows from \cite[Theorem 3.1]{DSW} that
if $\left( \mathcal{F}_{n}\right) _{n}$ and $\left( \mathcal{G}_{n}\right)
_{n}$ are sequences of filters and $\left( X_{n}\right) _{n}$ is a sequence
of disjoint sets such that%
\begin{equation}
\underset{n<\omega }{\forall }\;X_{n}\in \mathcal{F}_{n}\cap \mathcal{G}_{n},
\label{t}
\end{equation}%
then $\limfunc{Li}\nolimits_{\left( n\right) }\mathcal{F}_{n}=\limfunc{Li}%
\nolimits_{\left( n\right) }\mathcal{G}_{n}$ implies $\mathcal{F}_{n}=%
\mathcal{G}_{n}$ for almost all $n.$ This is, a particularly simple case,
but it has a generalization without the assumption (\ref{t}) (Theorem \ref%
{cor:trans}). In general, however, the relation $\limfunc{Li}%
\nolimits_{\left( k\right) }\mathcal{G}_{k}\geq \limfunc{Li}%
\nolimits_{\left( n\right) }\mathcal{F}_{n}$ does not even imply that $%
\mathcal{G}_{k}$ and $\mathcal{F}_{n}$ \emph{mesh} (for the definition, see
Section \ref{sec:grill} below) for some $\left( n,k\right) $.

The \emph{Alternative} (Theorem \ref{thm:alt}) is a key to the contour
inversion problem. We apply the Alternative to \emph{sequential contours},
that is, to iterated contours of \emph{sequential cascades,} to be defined
in Section \ref{sec:casc}. The contour of a sequential cascade is, in fact,
the restriction to its maximal elements of the topological modification of
the original convergence of the cascade. Therefore, in this case, contour
inversion is akin to \emph{continuous extension} of maps between sequential
cascades. We show that if $V,W$ are monotone sequential cascades, then for
each continuous map $\varphi :\limfunc{ext}V\rightarrow \limfunc{ext}W$
(with respect to the standard topologies of sequential cascades), there
exists a full, closed downwards subcascade $T$ of $V$ and a continuous map $%
f:T\rightarrow W$ such that $\left. f\right\vert _{\limfunc{ext}T}=\left.
\varphi \right\vert _{\limfunc{ext}T}$. This fact was formulated in \cite[%
Theorem 3.3]{DSW} in a rather special case of injective maps $\varphi $ and
repeated in \cite{multisequences} with the same proof that turns out not to
be correct.

A topology is called \emph{subsequential} if it is a subspace of a
sequential topology. In \cite{DW.subseq}\cite{multisequences}, subsequential
spaces were characterized by the following equivalent statements:

\begin{enumerate}
\item A topological space $X$ is subsequential,

\item $x\in \limfunc{cl}\nolimits_{X}A$ if and only if there exists a
sequential contour $\mathcal{F}$ on $A$ such that $x\in \lim\nolimits_{X}%
\mathcal{F}$,

\item $A$ is $X$-closed if and only if $\lim\nolimits_{X}\mathcal{F}\subset
A $ for each sequential contour $\mathcal{F}$ on $A.$
\end{enumerate}

This shows the importance of sequential contours in the theory of
subsequential spaces (\cite{DSW},\cite{FrRaj}). In particular, in \cite{DSW}%
, a generating role of \emph{supercontours,} that is, of the suprema $%
\bigvee_{\alpha <\omega _{1}}\mathcal{F}_{a}$ where $\mathcal{F}_{a}$ is a
sequential contour of rank $\alpha ,$ was studied. It was shown that if a
filter $\mathcal{H}$ generates every subsequential topology, then $\mathcal{H%
}$ is not a sequential contour. On the other hand, under $\emph{CH}$, there
exists a supercontour $\mathcal{K}$ such that the class generated by $%
\mathcal{K}$ strictly includes the class of subsequential topologies.

The continuous extension applies as well in set-theory and logic for some
classifications of ultrafilters on natural numbers, for example, to
establish a hierarchy of $\omega _{1}$-sequence of disjoint classes of
ultrafilters on natural numbers with respect to the \emph{level of
complication}. Specifically, an ultrafilter $u\in P_{\alpha }$ if for each $%
\beta <\alpha $ there exists a monotone sequential contour $\mathcal{C}$ of
rank $\beta $, such that $\mathcal{C}\subset u$ and there is no monotone
sequential contour of rank $\alpha $ contained in $u$. It appears that the
class $P_{2}$ is precisely that of $P$-points (\cite{StaroBeta},\cite{StaroP}%
,\cite{MachStar}). The continuous extension (Theorem \ref{thm:ext}) implies
that if $u\in P_{\alpha }$ then there is $\gamma \leq \alpha $ such that $%
f\left( u\right) \in P_{\gamma }$. Monotone sequential cascades are used in
\cite{StarosolskiCOU} to solve Baumgartner's problem \cite{baum}, in the
first nontrivial case, showing (ZFC) that the class of proper $J_{\omega
^{\omega }}$-ultrafilters is empty. It is equally proved in \cite%
{StarosolskiCOU} that if there is an arbitrarily long finite $<_{\infty }$%
-sequence (in Laflamme's terminology \cite{laflamme}) under $u$, then $u$ is
at least a strict $J_{\omega ^{\omega +1}}$-ultrafilter.

\section{Grill\label{sec:grill}}

Recall that families of sets $\mathcal{A}_{1},\mathcal{A}_{2},\ldots ,%
\mathcal{A}_{n}$ \emph{mesh}, in symbols,%
\begin{equation*}
\mathcal{A}_{1}\#\mathcal{A}_{2}\ldots \#\mathcal{A}_{n}
\end{equation*}%
if $A_{1}\cap A_{2}\cap \ldots \cap A_{n}\neq \varnothing $ for each $%
A_{1}\in \mathcal{A}_{1},A_{2}\in \mathcal{A}_{2},\ldots $,$A_{n}\in
\mathcal{A}_{n}$ \cite{CFT}. We say that $\mathcal{A},\mathcal{B}$ are \emph{%
disjoint} (\footnote{%
Or, \emph{dissociated} \cite{CFT}.}) if $\mathcal{A}\lnot \#\mathcal{B}$. We
abridge $\left\{ A\right\} \#\mathcal{B}$ to $A\#\mathcal{B}$ and even $%
\left\{ A\right\} \#\left\{ B\right\} $ to $A\#B.$ The \emph{grill }$%
\mathcal{A}^{\#}$ of a family $\mathcal{A}$ of subsets of $X$ is defined by%
\begin{equation*}
\mathcal{A}^{\#}:=\left\{ H\subset X:\underset{A\in \mathcal{A}}{\forall }%
\;H\#A\right\} .
\end{equation*}

A sequence $\left( \mathcal{F}_{n}\right) _{n}$ of filters is called \emph{%
disjoint} if $\mathcal{F}_{n}\lnot \#\mathcal{F}_{m}$ for all $n\neq m;$
\emph{totally disjoint} if there exists a sequence $\left( X_{n}\right) _{n}$
of disjoint sets such that $X_{n}\in \mathcal{F}_{n}$ for each $n<\omega .$

\begin{lemmaA}
A sequence of filters $\left( \mathcal{F}_{n}\right) _{n}$ is totally
disjoint if and only if $\left( \mathcal{F}_{n}\right) _{n}$ is disjoint and
$\limfunc{Li}\nolimits_{\left( k\right) }\mathcal{F}_{k}$ is disjoint from $%
\mathcal{F}_{n}$ for each $n<\omega $.
\end{lemmaA}

\begin{proof}
Let $\left( X_{n}\right) _{n}$ be a sequence of disjoint sets such that $%
X_{n}\in \mathcal{F}_{n}$ for each $n<\omega .$ Then $X_{n}$ and $%
\bigcup_{n\neq k<\omega }X_{k}$ are disjoint and $\bigcup_{n\neq k<\omega
}X_{k}\in \limfunc{Li}\nolimits_{\left( k\right) }\mathcal{F}_{k},$ so that $%
\mathcal{F}_{n}$ is disjoint from $\limfunc{Li}\nolimits_{\left( k\right) }%
\mathcal{F}_{k}.$ Conversely, if $\mathcal{F}_{n}$ is disjoint from $%
\limfunc{Li}\nolimits_{\left( k\right) }\mathcal{F}_{k},$ that is, by
definition, there is $k_{n}>n$ such that $X\setminus F_{n}\in \mathcal{F}%
_{k} $ for each $k\geq k_{n},$ then $F_{n}\in \mathcal{F}_{n}$ and $%
X\setminus F_{n}\in \mathcal{F}_{k}$ for each $k\geq k_{n}.$ If, moreover, $%
\left( \mathcal{F}_{n}\right) _{n}$ is disjoint, then also $X\setminus
F_{n}\in \mathcal{F}_{k}$ for each $k<k_{n}.$ Define $X_{0}:=F_{0}$ and,
inductively, $X_{n}:=F_{n}\setminus \bigcup_{k=0}^{n-1}F_{k}$ for $n>0.$
Then $\left( X_{n}\right) _{n}$ is a sequence of disjoint sets such that $%
X_{n}\in \mathcal{F}_{n}$ for every $n<\omega .$
\end{proof}

\section{The Alternative}

A relation $A\subset \left\{ \left( n,k\right) :n<\omega ,k<\omega \right\} $
is called \emph{transversal} if $A$ is infinite, and $\left\{ l:\left(
n,l\right) \in A\right\} $ and $\left\{ m:\left( m,k\right) \in A\right\} $
are at most singletons for each $n,k<\omega .$

\begin{theoremA}[Alternative]
\label{thm:alt}Let $\left( \mathcal{F}_{n}\right) _{n}$ and $\left( \mathcal{%
G}_{k}\right) _{k}$ be sequences of filters on a set $X$, and let%
\begin{equation*}
\mathcal{F}:=\limfunc{Li}\nolimits_{\left( n\right) }\mathcal{F}_{n}\text{
and }\mathcal{G}:=\limfunc{Li}\nolimits_{\left( k\right) }\mathcal{G}_{k}.
\end{equation*}%
If $\mathcal{F}\#\mathcal{G}$, then one of the following statements holds:

\begin{enumerate}
\item[A.1.] $\mathcal{F}_{n}\#\mathcal{G}_{k}$ for a transversal set of $%
\left( n,k\right) $,\label{A1}

\item[A.2.] $\mathcal{F}\#\mathcal{G}_{k}$ for infinitely many $k,$\label{A2}

\item[A.3.] $\mathcal{F}_{n}\#\mathcal{G}$ for infinitely many $n.$\label{A3}
\end{enumerate}
\end{theoremA}

\begin{proof}
We shall prove the alternative with (A.1) replaced by

\begin{enumerate}
\item[A.0.] for each $m$ there there exists $\left( n,k\right) $ such that $%
m<\min \left\{ n,k\right\} $ and $\mathcal{F}_{n}\#\mathcal{G}_{k}.$
\end{enumerate}

If none of the cases (A.0)(A.2)(A.3) holds, then there is $m<\omega $ such
that%
\begin{eqnarray}
&&\mathcal{F}_{n}\lnot \#\mathcal{G}_{k},  \label{1} \\
&&\mathcal{F}\lnot \#\mathcal{G}_{k},  \label{2} \\
&&\mathcal{F}_{n}\lnot \#\mathcal{G}  \label{3}
\end{eqnarray}%
for each $n,k>m.$ As%
\begin{equation*}
\mathcal{F}=\limfunc{Li}\nolimits_{\left( n\right) _{n>m}}\mathcal{F}_{n}%
\text{ and }\mathcal{G}=\limfunc{Li}\nolimits_{\left( k\right) _{k>m}}%
\mathcal{G}_{k}
\end{equation*}%
for each $m<\omega ,$ we can assume without loss of generality that (\ref{1}%
)-(\ref{3}) hold for each $n,k<\omega .$

Consequently, by (\ref{2}), for each $k,$ there exist $F_{\infty ,k}\in
\mathcal{F}$ and $G_{\infty ,k}\in \mathcal{G}_{k}$ such that $F_{\infty
,k}\cap G_{\infty ,k}=\varnothing .$ By the definition of contour, for each $%
k$ there is $n_{k}>k$ such that $F_{\infty ,k}\in \mathcal{F}_{n}$ for $%
n>n_{k}$. Analogously, for each there exist $F_{n,\infty }\in \mathcal{F}%
_{n} $ and $G_{n,\infty }\in \mathcal{G}$ such that $F_{n,\infty }\cap
G_{n,\infty }=\varnothing .$ By the definition of contour, for each $n$
there exists $k_{n}>n$ such that $G_{n,\infty }\in \mathcal{G}_{k}$ for each
$k>k_{n}.$ By (\ref{1}), for each $\left( n,k\right) $ there exist $%
X_{n,k}\in \mathcal{F}_{n}$ and $Y_{n,k}\in \mathcal{G}_{k}$ such that $%
X_{n,k}\cap Y_{n,k}=\varnothing .$ Let%
\begin{gather*}
F_{n,k}:=F_{\infty ,k}\text{ and }G_{n,k}:=G_{\infty ,k}\text{ if }n>n_{k},
\\
F_{n,k}:=F_{n,\infty }\text{ and }G_{n,k}:=G_{n,\infty }\text{ if }k>k_{k},
\\
F_{n,k}:=X_{n,\infty }\text{ and }G_{n,k}:=Y_{n,\infty },\text{ otherwise.}
\end{gather*}

Finally, define%
\begin{equation*}
F_{n}:=\bigcap\nolimits_{k<\omega }F_{n,k}\text{ and }G_{k}:=\bigcap%
\nolimits_{n<\omega }G_{n,k},
\end{equation*}%
and notice that the intersections above are finite (!), hence $F_{n}\in
\mathcal{F}_{n}$ and $G_{k}\in \mathcal{G}_{k},$ and $F_{n}\cap
G_{k}=\varnothing $ for each $\left( n,k\right) .$ If%
\begin{equation*}
F:=\bigcup\nolimits_{n<\omega }F_{n}\text{ and }G:=\bigcup\nolimits_{k<%
\omega }G_{k},
\end{equation*}%
then $F\in \mathcal{F}$ and $G\in \mathcal{G}$ and $F\cap G=\varnothing $,
which completes the proof.
\end{proof}

\begin{exampleA}[Only the first case holds]
\label{ex:only1}Let $\left( \mathcal{F}_{n}\right) _{n}$ be a totally
disjoint and let $\left( \mathcal{G}_{k}\right) _{k}$ be such that $\mathcal{%
F}_{n}\leq \mathcal{G}_{n}$ for each $n<\omega .$ Then $\mathcal{F}\leq
\mathcal{G}$ hence $\mathcal{F}\#\mathcal{G}.$ As $\left( \mathcal{F}%
_{n}\right) _{n}$ is totally disjoint, $\mathcal{F}_{n}\lnot \#\mathcal{F}$,
hence $\mathcal{F}_{n}\lnot \#\mathcal{G}$ thus $\mathcal{G}_{n}\lnot \#%
\mathcal{F}$ for each $n<\omega .$ Therefore, only (A.1) holds. More
precisely, $\mathcal{F}_{n}\#\mathcal{G}_{n},\;\mathcal{F}_{n}\lnot \#%
\mathcal{G},\;\mathcal{G}_{n}\lnot \#\mathcal{F}$ for each $n<\omega .$
\end{exampleA}

\begin{exampleA}[Only the last case holds]
\label{ex:only}(\footnote{%
A similar example has been communicated to the first author by Tsugunori
Nogura as a counter-example to a certain hypothesis.}) Let $\left( \mathcal{F%
}_{n}\right) _{n}$ be a sequence of totally disjoint filters and let $\left(
\mathcal{G}_{k}\right) _{k}$ be such that $\mathcal{G}_{k}\geq \mathcal{F}$.
Then $\mathcal{F}\lnot \#\mathcal{F}_{n}$ and $\mathcal{G}\geq \mathcal{F}$,
hence $\mathcal{G}_{k}\lnot \#\mathcal{F}_{n}$ and $\mathcal{G}\lnot \#%
\mathcal{F}_{n}$ for each $n,k<\omega .$ Hence, only (A.3) holds. More
precisely, $\mathcal{G}_{k}\lnot \#\mathcal{F}_{n},\;\mathcal{G}\lnot \#%
\mathcal{F}_{n},\;\mathcal{G}_{k}\#\mathcal{F}$ for each $n,k<\omega .$
\end{exampleA}

Of course, by exchanging the roles of the sequences, one gets the case where
only (A.2) holds. Here is an example where all the cases hold.

\begin{exampleA}[All cases hold simultaneously]
Let $\mathcal{F}_{n}\leq \mathcal{F}_{n+1}\leq \ldots \leq \mathcal{G}%
_{k}\leq \mathcal{G}_{k+1}.$ Then $\mathcal{F}_{n}\leq \mathcal{G}_{k}$
(hence $\mathcal{F}_{n}\#\mathcal{G}_{k}$) for each $n,k<\omega $. Also $%
\mathcal{F}\leq \mathcal{G}_{k}$ hence $\mathcal{F}\#\mathcal{G}_{k}$ for
each $k<\omega ,$ and $\mathcal{F}_{n}\leq \mathcal{G}$ hence $\mathcal{F}%
_{n}\#\mathcal{G}$ for each $n<\omega ,$ and $\mathcal{F}\#\mathcal{G}$.
\end{exampleA}

It turns out that all possible variants of the Alternative can hold. More
precisely, if $\varnothing \neq J\subset \left\{ 1,2,3\right\} ,$ then there
exist sequences of filters $\left( \mathcal{F}_{n}\right) _{n}$ and $\left(
\mathcal{G}_{k}\right) _{k}$ such that ($A.j$) holds for $j\in J$ and does
not hold for $j\notin J.$ This can be done even so that the conditions hold
for all the indices, in the sense that will be made precise below.

\begin{propositionA}
If $\varnothing \neq J\subset \left\{ 1,2,3\right\} ,$ then there exist
sequences of filters $\left( \mathcal{F}_{n}\right) _{n}$ and $\left(
\mathcal{G}_{k}\right) _{k}$ such that (A.j) holds for $j\in J$ and does not
hold for $j\notin J$ for every $n=k<\omega .$
\end{propositionA}

\begin{proof}
We have seen (Examples \ref{ex:only}, \ref{ex:only1} and the subsequent
comment) that for each $j\in \left\{ 1,2,3\right\} ,$ there exist $\left(
\mathcal{F}_{j,n}\right) _{n}$ and $\left( \mathcal{G}_{j,k}\right) _{k}$ on
$X_{j}$ such that (A.j) holds for each $n=k<\omega ,$ and (A.i) does not
hold for every $n,k<\omega $ for $i\neq j.$ We assume that $X_{j}\cap
X_{k}=\varnothing $ if $j\neq k$ and $j,k\in J.$ Then%
\begin{eqnarray*}
\mathcal{F} &=&\limfunc{Li}\nolimits_{\left( n\right) }\mathcal{F}_{n}=%
\limfunc{Li}\nolimits_{\left( n\right) }\bigcap\nolimits_{j\in J}\mathcal{F}%
_{j,n}=\bigcap\nolimits_{j\in J}\limfunc{Li}\nolimits_{\left( n\right) }%
\mathcal{F}_{j,n}, \\
\mathcal{G} &=&\limfunc{Li}\nolimits_{\left( k\right) }\mathcal{G}_{k}=%
\limfunc{Li}\nolimits_{\left( k\right) }\bigcap\nolimits_{j\in J}\mathcal{G}%
_{j,k}=\bigcap\nolimits_{j\in J}\limfunc{Li}\nolimits_{\left( k\right) }%
\mathcal{G}_{j,k},
\end{eqnarray*}%
fulfill the condition.
\end{proof}

Here we get a first answer to our contour inversion problem:

\begin{theoremA}
\label{cor:trans}If $\left( \mathcal{F}_{n}\right) _{n}$ and $\left(
\mathcal{G}_{k}\right) _{k}$ are totally disjoint sequences and if%
\begin{equation}
\mathcal{H}:=\limfunc{Li}\nolimits_{\left( n\right) }\mathcal{F}_{n}=%
\limfunc{Li}\nolimits_{\left( k\right) }\mathcal{G}_{k},  \label{z}
\end{equation}%
then the relation%
\begin{equation}
\Delta :=\left\{ \left( n,k\right) :\mathcal{F}_{n}\#\mathcal{G}_{k}\right\}
\label{delta}
\end{equation}%
is finite-to finite with cofinite domain and range, and
\begin{equation}
\mathcal{H}=\limfunc{Li}\nolimits_{\left( n,k\right) }\left( \mathcal{F}%
_{n}\vee \mathcal{G}_{k}\right) ,  \label{e}
\end{equation}%
where $\left( n,k\right) $ denotes the cofinite filter on $\Delta .$
Moreover, $\mathcal{F}_{n}=\bigcap\nolimits_{k\in \Delta n}\left( \mathcal{F}%
_{n}\vee \mathcal{G}_{k}\right) $ for almost all $n$ and $\mathcal{G}%
_{k}=\bigcap\nolimits_{n\in \Delta ^{-1}k}\left( \mathcal{F}_{n}\vee
\mathcal{G}_{k}\right) $ for almost all $k.$
\end{theoremA}

\begin{proof}
As $\left( \mathcal{F}_{n}\right) _{n}$ is totally disjoint, $\mathcal{H}%
\lnot \#\mathcal{F}_{n}$ for each $n.$ As $\left( \mathcal{G}_{k}\right)
_{k} $ is totally disjoint, $\mathcal{H}\lnot \#\mathcal{G}_{k}$ for each $%
k. $ Therefore, by the Alternative (Theorem \ref{thm:alt}), there is a
transversal subset of $\Delta .$ Actually, $\Delta $ itself is
finite-to-finite with cofinite domain and range. Indeed, if for some $n$ the
set $\Delta n$ were infinite, then $\limfunc{Li}\left\{ \mathcal{F}_{n}\vee
\mathcal{G}_{k}:k\in \Delta n\right\} \geq \mathcal{F}_{n}\lnot \#\mathcal{H}
$ and, on the other hand, $\limfunc{Li}\left\{ \mathcal{F}_{n}\vee \mathcal{G%
}_{k}:k\in \Delta n\right\} \geq \limfunc{Li}\left\{ \mathcal{G}_{k}:k\in
\Delta n\right\} \geq \mathcal{H}$, which yields a contradiction. The same
argument shows that $\Delta ^{-1}k$ is finite for each $k.$ It follows that%
\begin{equation*}
\limfunc{Li}\nolimits_{\left( n,k\right) }\left( \mathcal{F}_{n}\vee
\mathcal{G}_{k}\right) \geq \mathcal{H}.
\end{equation*}%
Suppose that $A:=\left\{ n:\Delta n=\varnothing \right\} $ is infinite. Then
$\limfunc{Li}\nolimits_{n\in A}\mathcal{F}_{n}\geq \mathcal{H}$ and $%
\mathcal{F}_{n}\lnot \#\mathcal{H},$ hence by Theorem \ref{thm:alt}, there
is a transversal subset of%
\begin{equation*}
\left\{ \left( n,k\right) \in \Delta :n\in A\right\} ,
\end{equation*}%
which yields a contradiction. Alike, one shows that $\left\{ k:\Delta
^{-}k=\varnothing \right\} $ is finite.

Consequently, if $A\in \limfunc{Li}\nolimits_{\left( n,k\right) }\left(
\mathcal{F}_{n}\vee \mathcal{G}_{k}\right) ,$ then $A\in \mathcal{F}_{n}\cap
\mathcal{G}_{k}$ for a cofinite subset of $\Delta ,$ hence $A\in \mathcal{F}%
_{n}\cap \mathcal{G}_{k}$ for almost all $n$ (and almost all $k$), which
implies that $A\in \mathcal{H}$, hence (\ref{e}).

Let $\left\{ X_{n}:n<\omega \right\} $ and $\left\{ Y_{k}:k<\omega \right\} $
be sequences of disjoint sets such that $X_{n}\in \mathcal{F}_{n}$ and $%
Y_{k}\in \mathcal{G}_{k}$ for all $n,k<\omega .$ By definition, $%
\bigcup\nolimits_{\left( n,k\right) \in \Delta }\left( X_{n}\cap
Y_{k}\right) \in \mathcal{H}$. Suppose that there is an infinite set $N$
such that $\mathcal{F}_{n}\ngeqslant \bigcap\nolimits_{k\in \Delta n}\left(
\mathcal{F}_{n}\vee \mathcal{G}_{k}\right) $ for all $n\in N,$ then there
exist $U_{n}\subset X_{n}$ such that $U_{n}\in \bigcap\nolimits_{k\in \Delta
n}\left( \mathcal{F}_{n}\vee \mathcal{G}_{k}\right) \setminus \mathcal{F}%
_{n} $ for each $n\in N.$ Therefore $U:=\bigcup\nolimits_{n\in N}U_{n}\cup
\bigcup\nolimits_{n\notin N}X_{n}\in \mathcal{H}$ from (\ref{e}), but $%
U\notin \limfunc{Li}\nolimits_{\left( n\right) }\mathcal{F}_{n},$ which is a
contradiction with (\ref{z}).
\end{proof}

Let $\left( \mathcal{F}_{n}\right) _{n}$ and $\left( \mathcal{H}_{p}\right)
_{p}$ be sequences of filters on a set $X$. Then $\left( \mathcal{H}%
_{p}\right) _{p}$ is called a \emph{locally finite refinement} of $\left(
\mathcal{F}_{n}\right) _{n}$ if there exists a map $f:\omega \rightarrow
\omega $ with finite fibers such that%
\begin{equation*}
\mathcal{F}_{n}=\bigcap\nolimits_{f\left( p\right) =n}\mathcal{H}_{p}
\end{equation*}%
for almost all $n.$ A locally finite refinement is called a \emph{locally
finite partition} if $\left( \mathcal{H}_{p}\right) _{p}$ is disjoint.

\begin{corollaryA}
If two totally disjoint sequences of filters have the same contour, then
this is also the contour of a their common locally finite refinement.
\end{corollaryA}

Actually, a least common refinement of totally disjoint sequences $\left(
\mathcal{F}_{n}\right) _{n}$ and $\left( \mathcal{G}_{k}\right) _{k}$ (for
which (\ref{delta}) is finite-to-finite with cofinite domain and range) is
given by%
\begin{equation*}
\mathcal{H}_{n,k}:=\mathcal{F}_{n}\vee \mathcal{G}_{k}.
\end{equation*}

In other words, although the contour of a sequence of filters does not
determine the sequence, but it does the class of its all locally finite
partitions. Therefore, two sequences of filters with the same contour are
intimately related: there exists a third sequence such that each of the two
consists of finite infima of the third.

\section{The Stone space interpretations}

Given a filter $\mathcal{H}$ on a set $X,$ we denote by $\beta \mathcal{H}$
the set of all ultrafilters finer than $\mathcal{H}$. The so defined map $%
\beta $ is an order isomorphism between the complete lattice of (possibly
degenerate) filters on $X$ and of the complete lattice of closed sets in the
Stone topology on $\beta X$ (the set of all ultrafilters on $X$). In
particular,%
\begin{equation*}
\beta \left( \bigcup\nolimits_{i\in I}\mathcal{F}_{i}\right)
=\bigcap\nolimits_{i\in I}\beta \mathcal{F}_{i}\text{ and }\beta \left(
\bigcap\nolimits_{i\in I}\mathcal{F}_{i}\right) =\limfunc{cl}%
\nolimits_{\beta }\left( \bigcup\nolimits_{i\in I}\beta \mathcal{F}%
_{i}\right) .
\end{equation*}%
Therefore, the image by $\beta $ of the contour of $\left( \mathcal{F}%
_{n}\right) _{n}$ is equal to the \emph{upper limit\ }of $\left( \beta
\mathcal{F}_{n}\right) _{n}$ \emph{\ }(\footnote{%
If $\tau $ is a topology, then the upper $\tau $-limit of a sequence $\left(
A_{n}\right) _{n}$ of $\tau $-closed sets is defined by%
\begin{equation*}
\limfunc{Ls}\nolimits_{\left( n\right) }^{\tau }A_{n}:=\bigcap_{n<\omega }%
\limfunc{cl}\nolimits_{\tau }(\bigcup_{k>n}A_{k}).
\end{equation*}%
\par
It is commonly called the \emph{Kuratowski-Painlev\'{e} upper limit}.
However, Kuratowski defines upper and lower limits in 1928, citing a paper
of Painlev\'{e}, in which in fact the concept does not appear. In 1912 a
Painlev\'{e}'s pupil, Zoretti, testifies that Painlev\'{e} used the notion
in 1902. But the upper limit appears in the version 1902-1903 of \emph{%
Formulaire Math\'{e}matique} of Peano, who certainly knew it well before, as
he defined formally the \emph{lower limit} in \emph{Applicazioni geometriche}
in 1887 (see \cite{Tang}).}), namely%
\begin{equation*}
\beta (\limfunc{Li}\nolimits_{\left( n\right) }\mathcal{F}%
_{n})=\bigcap_{n<\omega }\limfunc{cl}\nolimits_{\beta }(\bigcup_{k>n}\beta
\mathcal{F}_{k})=\limfunc{Ls}\nolimits_{\left( n\right) }^{\beta }\beta
\mathcal{F}_{n}.
\end{equation*}

Consequently, in terms of the Stone topology, Theorem \ref{thm:alt} becomes

\begin{theoremA}[Alternative]
\label{thm:alt copy(1)}Let $\left\{ A_{n}:n<\omega \right\} $ and $\left\{
B_{k}:k<\omega \right\} $ be sequences of $\beta $-closed sets. If $\left(
\limfunc{Ls}\nolimits_{\left( n\right) }^{\beta }A_{n}\right) \cap \left(
\limfunc{Ls}\nolimits_{\left( k\right) }^{\beta }B_{k}\right) \neq
\varnothing ,$ then one of the following statements holds:

\begin{enumerate}
\item[A.1.] $A_{n}\cap B_{k}\neq \varnothing $ for a transversal set of $%
\left( n,k\right) $,

\item[A.2.] $A\cap B_{k}\neq \varnothing $ for infinitely many $k,$

\item[A.3.] $A_{n}\cap B\neq \varnothing $ for infinitely many $n.$
\end{enumerate}
\end{theoremA}

This means that if two upper limits in the Stone topology have non-empty
intersection, then at least infinitely many terms of one sequence meet
either another upper limit, or infinitely many terms of another sequence.
Therefore, the behavior of upper limits with respect to the Stone topology
is rather peculiar. In particular, a dual of Theorem \ref{cor:trans} reads
as follows (we say that a sequence $\left( H_{n}\right) _{n}$ of closed $%
\beta $-closed sets is \emph{totally disjoint} if there is a sequence $%
\left( U_{n}\right) _{n}$ of disjoint $\beta $-open sets (\footnote{%
As each $\beta $-closed set is $\beta $-compact, we can replace in this
statement \emph{open} by \emph{clopen}.}) such that $H_{n}\subset U_{n}$ for
each $n<\omega $)$.$

\begin{theoremA}
If $\left( A_{n}\right) _{n}$ and $\left( B_{k}\right) _{k}$ are totally
disjoint sequences of $\beta $-closed sets and $H:=\limfunc{Ls}%
\nolimits_{\left( n\right) }^{\beta }A_{n}=\limfunc{Ls}\nolimits_{\left(
k\right) }^{\beta }B_{k}$, then%
\begin{equation*}
H=\limfunc{Ls}\nolimits_{\left( n,k\right) }^{\beta }\left( A_{n}\cap
B_{k}\right) ,
\end{equation*}%
where $\left( n,k\right) $ denotes the cofinite filter on $\Gamma :=\left\{
\left( n,k\right) :A_{n}\cap B_{k}\neq \varnothing \right\} .$ Moreover, $%
A_{n}=\bigcup_{k\in \Gamma n}\left( A_{n}\cap B_{k}\right) $ and $%
B_{k}=\bigcup_{n\in \Gamma ^{-1}k}\left( A_{n}\cap B_{k}\right) .$
\end{theoremA}

In particular (\footnote{%
We are grateful to professor T. Nogura for having pointed out a mistake in a
preliminary version of this paper.}),

\begin{corollary}
If $A$ and $B$ are countable sets of totally disjoint ultrafilters such that%
\begin{equation*}
\partial _{\beta }A=\partial _{\beta }B,
\end{equation*}%
then $A$ and $B$ are almost equal.
\end{corollary}

The condition of being totally disjoint cannot be dropped.

\begin{exampleA}
Let $\left( p_{n}\right) _{n}$ be a totally disjoint sequence of
ultrafilters. Hence $\partial _{\beta }\left\{ p_{m}:m<\omega \right\} \cap
\left\{ p_{n}:n<\omega \right\} =\varnothing .$ Let $\left( q_{k}\right)
_{k} $ be a sequence of distinct ultrafilters such that $q_{2k}=p_{k}$ and $%
q_{2k+1}\in \partial _{\beta }\left\{ p_{m}:m<\omega \right\} $ for each $%
k<\omega .$ Then $\partial _{\beta }\left\{ q_{2k+1}:k<\omega \right\}
\subset \partial _{\beta }\left\{ p_{m}:m<\omega \right\} $%
\begin{equation*}
\partial _{\beta }\left\{ p_{m}:m<\omega \right\} =\partial _{\beta }\left\{
q_{2k}:k<\omega \right\} =\partial _{\beta }\left\{ q_{k}:k<\omega \right\} ,
\end{equation*}%
but $\left\{ q_{k}:k<\omega \right\} \setminus \left\{ p_{m}:m<\omega
\right\} $ is infinite.
\end{exampleA}

\section{Sequential cascades and contours\label{sec:casc}}

A \emph{cascade} $\left( V,\sqsubset \right) $ is a tree with a least
element $\varnothing _{V},$ well-founded for the inverse order (\footnote{%
That is, the set $\max A,$ of maximal elements of each non empty subset $A$
of $T$, is nonempty}). A cascade is called \emph{sequential} if for every $%
v\in V\setminus \max V$, the set $V^{+}\left( v\right) $ of immediate
successors of $v$ is countably infinite (\footnote{%
See S.\ Dolecki, F. Mynard \cite{cascades}, S.\ Dolecki \cite{multisequences}%
, S.\ Dolecki, A. Starosolski, S. Watson \cite{DSW}.}).

Each cascade is order-isomorphic to a full, closed downwards subtree of the
\emph{sequential tree} $\Sigma $ (\footnote{%
That is, the set of finite sequences of natural numbers (the empty set is
denoted by $\varnothing $). If $s$ and $t$ are elements of $\Sigma $, then
the \emph{concatenation} of $s$ and $t$ is denoted by $s\frown t$. The
abbreviation $(s,n)$ for $s\frown (n)$ (where $s\in \Sigma $ and $n<\omega $%
) is a useful abuse of notation. By definition, $s<t$ if there is a
non-empty finite sequence $r$ such that $t=s\frown r$. With so defined
partial order $\Sigma $ becomes a tree.})(\footnote{%
A subset $T$ of a partially ordered set $\left( \Sigma ,\sqsubset \right) $
is \emph{closed downwards} if $v\sqsubset s\in T$ implies $v\in T$.})(%
\footnote{%
A subset $T$ of $\Sigma $ is called \emph{full} if $T\cap \Sigma ^{+}(s)\neq
\varnothing $ implies that $\Sigma ^{+}(s)\subset T$ for every $s\in \Sigma $%
.}). The standard convergence $\sigma $ on $\Sigma $ (\footnote{%
By a convergence on $X$ we understand a relation $x\in \lim \mathcal{F}$
between filters $\mathcal{F}$ on $X$ and elements $x$ of $X$ such that $%
\mathcal{F}\leq \mathcal{G}$ implies $\lim \mathcal{F}\subset \lim \mathcal{G%
}$, and $x\in \lim \{x\}^{\uparrow }$ where $\{x\}^{\uparrow }$ stands for
the principal ultrafilter of $x$. If $\mathcal{B}$ is a filter base, then we
often abridge $\lim \mathcal{B}\ $for the limit of the filter generated by $%
\mathcal{B}$.}) is that for which the cofinite filter on $\Sigma ^{+}\left(
s\right) $ converges to $s\ $(\footnote{%
The \emph{cofinite filter} $\left( A\right) _{0}$ of an infinite subset $A$
of a set $X$ is defined in \cite{CFT} by%
\begin{equation*}
\left( A\right) _{0}:=\left\{ F\subset X:\limfunc{card}\left( A\setminus
F\right) <\aleph _{0}\right\} .
\end{equation*}%
})$.$ The finest topology $\func{T}\sigma $ compatible with $\sigma $ is
called the \emph{standard topology.} The standard convergence $\left. \sigma
\right\vert _{V}$ induced on a sequential cascade $V$ and the induced
standard topology fulfill%
\begin{equation*}
T\left( \left. \sigma \right\vert _{V}\right) =\left. T\sigma \right\vert
_{V}.
\end{equation*}

Let $V$ be a sequential cascade. The elements of $\limfunc{ext}%
V=\{\varnothing _{V}\}\cup \max V$ are called \emph{extremal} and the
restriction of the standard topology of $V$ to $\limfunc{ext}V$ will be
called an \emph{Arens topology}. Each Arens topology is \emph{prime}, that
is, such that there is at most one non isolated element. Such spaces were
applied in the study of subsequential topologies by Franklin and Rajagopalan
in \cite{FrRaj}. The \emph{rank} $r_{V}(v)$ of $v\in V$ is defined
inductively to be $0$ if $v\in \max V$, and otherwise the least ordinal
greater than the ranks of the successors of $v$. The rank $r(V)$ of a
sequential cascade $V$ is by definition the rank of $\varnothing _{V}$. It
is always a countable ordinal.

The \emph{level} $l_{V}(\varnothing )=0$ and for $v\neq \varnothing _{V}$
the \emph{level }of $v$ is defined by
\begin{equation*}
l_{V}(v)=\max \{l_{V}(s)+1:s\sqsubset v\}.
\end{equation*}%
As $\left\{ s\in V:s\sqsubset v\right\} $ is well ordered, the level $%
l_{V}(v)$ is finite for every $v\in V.$ The set%
\begin{equation*}
\limfunc{sp}\nolimits_{V\ }\left( v\right) :=\left\{ r_{V}\left( s\right)
:s\sqsubseteq v\right\}
\end{equation*}%
is called the \emph{spectrum} of $v$ in $V.$ It is helpful to have in mind
that each element of a cascade has finite spectrum.

The \emph{(sequential) contour} $\int V$ of a sequential cascade $V$ is
defined, by induction. If $r\left( V\right) =0,$ then $\int V$ is the
principal filter of $\varnothing _{V};$ otherwise,%
\begin{equation}
\int V:=\limfunc{Li}\nolimits_{\left( V^{+}\left( \varnothing _{V}\right)
\right) _{0}}\left( \int V\left( v\right) \right) ,  \label{concasc}
\end{equation}%
where $V\left( v\right) $ is the cascade formed by all the successors of $%
v\in V,$ and the exterior contour is taken over the cofinite filter $\left(
V^{+}\left( \varnothing _{V}\right) \right) _{0}$ on the set $V^{+}\left(
\varnothing _{V}\right) ,$ of the immediate successors of $\varnothing _{V}$%
. It turns out that $\int V$ coincides with the trace on $\max V$ of the
neighborhood filter of $\varnothing _{V}$ with respect to the standard
topology, that is, with the neighborhood filter of $\varnothing _{V}$ of the
corresponding Arens topology \cite{cascades}. The \emph{rank} of a
sequential contour $\mathcal{F}$ on $X$ is the least ordinal $r(V)$ of a
sequential cascade $V$ such that $\mathcal{F}=\int V.$

A sequential cascade $V$ is called \emph{monotone} if for each $v\in
V\setminus \max V$ the set $V^{+}\left( v\right) ,$ of immediate successors
of $v,$ admits an order of the type $\omega _{0}$ for which $r_{V}$ is
non-decreasing. A sequential cascade $V$ is called \emph{asymptotically
monotone} if the rank function $r_{V}$ is lower semicontinuous. A sequential
contour $\mathcal{F}$ is called \emph{monotone} if there is a monotone
(equivalently, an asymptotically monotone) sequential cascade $V$ such that $%
\mathcal{F}=\int V$ (\footnote{%
Each asymptotically monotone sequential cascade is homeomorphic to a \emph{%
monotone }sequential cascade \cite{DSW}\cite{multisequences}.}). Monotone
sequential contours have particularly simple structure among sequential
contours.

\subsection{Notation}

Let $\left( V,\sqsubset \right) $ be a cascade. If $I\subset \left[
0,r\left( V\right) \right] ,$ then we write%
\begin{equation*}
V^{I}:=\left\{ v\in V:r_{V}\left( v\right) \in I\right\} ,
\end{equation*}%
and, in particular, $V^{\alpha }=\left\{ v\in V:r_{V}\left( v\right) =\alpha
\right\} .$

We define the relations (\footnote{%
The notation is different from that used in \cite{StarosolskiCOU}.})%
\begin{equation*}
V^{\uparrow }:=\left\{ \left( v,x\right) \in V\times V:v\sqsubseteq
x\right\} \text{ and }V^{\downarrow }:=\left\{ \left( v,x\right) \in V\times
V:v\sqsupseteq x\right\} .
\end{equation*}%
Of course, $V^{\downarrow }=\left( V^{\uparrow }\right) ^{-},$ that is, $%
V^{\downarrow }$ is the inverse relation of $V^{\uparrow }$. As it is
customary, the image of an element $v$ of $V$ by $V^{\uparrow }$ is $%
V^{\uparrow }\left( v\right) ,$ and $V^{\uparrow }\left[ A\right] $ denotes
the image of a set $A$ by $V^{\uparrow }.$ However, we shall use the
following abbreviations:%
\begin{eqnarray*}
V\left( v\right) &:&=V^{\uparrow }\left( v\right) =\left\{ x\in
V:v\sqsubseteq x\right\} ,\;V\left[ A\right] :=\bigcup\nolimits_{v\in
A}V\left( v\right) , \\
V^{-}\left( v\right) &:&=V^{\downarrow }\left( v\right) =\left\{ x\in
V:v\sqsupseteq x\right\} ,\;V^{-}\left[ A\right] :=\bigcup\nolimits_{v\in
A}V^{-}\left( v\right) .
\end{eqnarray*}%
If $V$ is a cascade and $S\subset \max V,$ then%
\begin{equation}
\left[ S_{\#}V\right] :=\left\{ v\in V:S\#\int V\left( v\right) \right\} .
\label{meshS}
\end{equation}

\begin{lemmaA}
Let $V$ be a monotone sequential cascade. If $S\#\int V,$ then there is a
biggest subset $S_{\infty }$ of $S$ such that%
\begin{equation*}
V_{S}:=V^{-1}\left[ S_{\infty }\right]
\end{equation*}%
is a monotone sequential cascade.
\end{lemmaA}

\begin{proof}
The condition $S\#\int V$ is equivalent to $\varnothing _{V}\in \left[
S_{\#}V\right] .$ As (\ref{concasc}), the set $\left[ S_{\#}V\right] _{1}:=%
\left[ S_{\#}V\right] \cap V^{+}\left( \varnothing _{V}\right) $ is infinite
if $\varnothing _{V}\notin \max ,$ that is, if $r\left( V\right) >0.$ By
induction, for every $v\in \left[ S_{\#}V\right] _{n}\setminus \max V,$ the
set $\left[ S_{\#}V\right] \cap V^{+}\left( v\right) $ is infinite. Let%
\begin{equation*}
\left[ S_{\#}V\right] _{n+1}:=\bigcup \left\{ \left[ S_{\#}V\right] \cap
V^{+}\left( v\right) :v\in \left[ S_{\#}V\right] _{n}\setminus \max
V\right\} .
\end{equation*}

Then%
\begin{equation*}
V_{S}:=\bigcup\nolimits_{n<\omega }\left[ S_{\#}V\right] _{n}
\end{equation*}%
is the required monotone sequential cascade and $S_{\infty }=\mathrm{max\,}%
V_{S}.$ Indeed, if $v\in \mathrm{max\,}V_{S},$ then there is $n<\omega $
such that $v\in \left[ S_{\#}V\right] _{n}\cap S.$ If $x\sqsubset v,$ then $%
\left[ S_{\#}V\right] \cap V^{+}\left( x\right) $ is infinite. If $v\in
\mathrm{max\,}V\setminus \mathrm{max\,}V_{S},$ then there exists $x\sqsubset
v$ such that $x\notin \left[ S_{\#}V\right] .$
\end{proof}

As $V_{S}$ is a relation, we shall write $V_{S}\left[ B\right] $ and $%
V_{S}^{-}\left[ B\right] $ for the image and the preimage of $B$ by $V_{S}.$

\section{Heredity of the grill relation of sequential cascades}

Observe that if $V$ and $W$ are monotone sequential cascades such that $\int
V\geq \int W,$ then $\int V\lnot \#\int W\left( w\right) $ for each $w\in
W\setminus \left\{ \varnothing _{W}\right\} .$

\begin{theoremA}
\label{thm:alt_seq}Let $V,W$ be monotone sequential cascades with $\max
V\subset X$ and $\max W\subset Y$ and let $\varphi :X\rightarrow Z,\psi
:Y\rightarrow Z.$ If%
\begin{equation}
\varphi (\int V)\#\psi (\int W),  \label{hypo}
\end{equation}%
then there exist $R\subset X$ and $S\subset Y$ such that for each $v\in
V_{R} $ there exists $w\in W_{S}$ for which%
\begin{equation}
\varphi \left( \int V_{R}\left( v\right) \right) \#\psi \left( \int
W_{S}\left( w\right) \right) ,  \label{aa}
\end{equation}%
and for each $w\in W_{S}$ there exists $v\in V_{R}$ such that (\ref{aa})
holds.
\end{theoremA}

\begin{proof}
In the proof, $V,W,V_{n},W_{k}$ and so on, stand for monotone sequential
cascades. Order $\limfunc{Ord}\times \limfunc{Ord}$ as follows:%
\begin{equation*}
\left( \alpha ,\beta \right) \prec \left( \delta ,\gamma \right)
\end{equation*}%
if either $\max \left( \alpha ,\beta \right) <\max \left( \delta ,\gamma
\right) $ or $\max \left( \alpha ,\beta \right) =\max \left( \delta ,\gamma
\right) $ and $\min \left( \alpha ,\beta \right) <\min \left( \delta ,\gamma
\right) .$ The relation $\prec $ is well-founded (that is, each non-empty
subset of $\limfunc{Ord}\times \limfunc{Ord}$ has an $\prec $-minimal
element, and $\left\{ \left( \alpha ,\beta \right) :\left( \alpha ,\beta
\right) \prec \left( \delta ,\gamma \right) \right\} $ is a set) and we use
it below for induction \cite[Theorem 25']{Jech}.

If $\min \{r(V),r(W)\}=0$, say $r(W)=0,$ then $\psi (\int W)$ is a principal
ultrafilter, say generated by $z$. Since (\ref{hypo}), $\left\{ z\right\}
\#\varphi \left( \int V\right) $ and so $\varphi ^{-1}\left( z\right) \#\int
V.$ It suffices to take $R=\max V_{\varphi ^{-1}\left( z\right) }.$

So let $r\left( V\right) =\alpha >0$ and $r\left( W\right) =\beta >0$ and
assume that the claim is true for $V^{\prime }$ and $W^{\prime }$ such that $%
\left( r\left( V^{\prime }\right) ,r\left( W^{\prime }\right) \right) \prec
\left( \alpha ,\beta \right) .$ Then $\varphi \left( \int V\right) =\limfunc{%
Li}\nolimits_{\left( n\right) }\varphi (\int V_{n})$ with $r\left(
V_{n}\right) <\alpha $ and $\psi \left( \int W\right) =\int_{\left( k\right)
}\psi (\int W_{k})$ with $r\left( W_{k}\right) <\beta $ for each $n,k<\omega
.$

If (\ref{hypo}) holds, then, by Theorem \ref{thm:alt}, one of the following
claims is true:

\begin{enumerate}
\item there is a transversal $A$ such that $\varphi $($\int V_{n})\#\psi
(\int W_{k})$ for each $\left( n,k\right) \in A,$

\item $\varphi $($\int V_{n})\#\psi (\int W)$ for infinitely many $n,$

\item $\psi (\int W_{k})\#\varphi (\int V)$ for infinitely many $k.$
\end{enumerate}

If (1) then $r\left( V_{n}\right) <\alpha $ and $r\left( W_{k}\right) <\beta
,$ hence $\left( r\left( V_{n}\right) ,r\left( W_{k}\right) \right) \prec
\left( \alpha ,\beta \right) $. By the inductive assumption, for each $%
\left( n,k\right) \in A$ there exist $R_{n,k},S_{n,k}$ such that for each $%
v\in V_{R_{n,k}}$ there is $w\in W_{S_{n,k}}$ and for each $w\in W_{S_{n,k}}$
there is $v\in V_{R_{n,k}}$ such that%
\begin{equation*}
\varphi (\int V_{R_{n,k}}\left( v\right) )\#\psi (\int W_{S_{n,k}}\left(
w\right) ).
\end{equation*}%
Then
\begin{equation*}
R:=\bigcup\nolimits_{\left( n,k\right) \in A}R_{n,k}\text{ and }%
S:=\bigcup\nolimits_{\left( n,k\right) \in A}S_{n,k}
\end{equation*}%
satisfy the claim.

If (2), let $B$ be the set of $n$ for which $\varphi (\int V_{n})\#\psi
(\int W).$ Then $\left( r\left( V_{n}\right) ,r\left( W\right) \right) \prec
\left( \alpha ,\beta \right) ,$ hence by the inductive assumption, for each $%
n\in B$ there exist infinite sets $R_{n},S_{n}$ such that such that for each
$v\in \left( V_{n}\right) _{R_{n}}$ there is $w\in W_{S_{n}}$ and for each $%
w\in W_{S_{n}}$ there is so $v\in \left( V_{n}\right) _{R_{n}}$ that%
\begin{equation*}
\varphi (\int \left( V_{n}\right) _{R_{n}}\left( v\right) )\#\psi (\int
W_{S_{n}}\left( w\right) ).
\end{equation*}%
Then%
\begin{equation*}
R:=\bigcup\nolimits_{n\in B}R_{n}\text{ and }S:=\bigcup\nolimits_{n\in
B}S_{n}
\end{equation*}%
satisfies the claim. If (3) then argue \emph{mutatis mutandis} as for (2).
\end{proof}

It should be noted that Theorem \ref{thm:alt_seq} is strictly stronger than
its unilateral variants like that in Corollary \ref{cor:her}. In fact,

\begin{exampleA}
Consider two copies $V,W$ of $\left\{ s\in \Sigma :l\left( s\right) \leq
2\right\} ,$ where $\Sigma $ stands for the sequential tree (\footnotemark[6]%
). Let $\varphi :\max V\rightarrow \max W$ be a bijection given by%
\begin{equation*}
\varphi \left( n,k\right) :=\left\{
\begin{array}{l}
\left( n-1,k+1\right) \text{ if }n>0, \\
\left( k,0\right) \text{ if }n=0.%
\end{array}%
\right.
\end{equation*}%
For $R=\max V,S=\max W,$ we have $V_{R}=V$ and $W_{S}=W.$ Then $\int V\#\int
W,$ $\int W\left( \left( n\right) \right) =\int W\left( \left( n+1\right)
\right) $ and $\int W\left( s\right) =\int W\left( \varphi ^{-1}\left(
s\right) \right) $ for $s\in \max W.$ However, $\int V\left( \left( 0\right)
\right) $ does not mesh $\int W\left( s\right) $ for each $l\left( s\right)
\leq 2.$
\end{exampleA}

\begin{corollaryA}
\label{cor:her}If $V,W$ are monotone sequential cascades such that%
\begin{equation*}
\varphi (\int V)\#\psi (\int W),
\end{equation*}%
then there exists $R\ $such that for each $v\in V_{R}$ there exists $w\in W$
for which%
\begin{equation}
\varphi \left( \int V_{R}\left( v\right) \right) \#\psi \left( \int W\left(
w\right) \right) .  \label{ab}
\end{equation}
\end{corollaryA}

The heredity property of the meshing relation, described in Corollary \ref%
{cor:her}, can be sharpened when $\varphi (\int V)\#\psi (\int W)$ is
strengthened to $\varphi (\int V)\geq \psi (\int W).$

\begin{theoremA}
\label{thm:sharp}Let $V,W$ be monotone sequential cascades and let $\varphi
:\max V\rightarrow \max W$ be such that $\varphi (\int V)\geq \int W.$ Then
there exists an (infinite) set $T$ such that for each $v\in V_{T}$ there
exists a unique $w\left( v\right) $ for which%
\begin{equation}
\varphi (\int V_{T}\left( v\right) )\#\int W\left( w\left( v\right) \right) .
\label{sharp}
\end{equation}
\end{theoremA}

\begin{proof}
By Corollary \ref{cor:her}, there exists a set $R$ such that for every $v\in
V_{R}$ there exists $w\in W$ such that (\ref{ab}) holds with $\psi $ being
the identity. To simplify, let $V:=V_{R}$ and then the condition becomes:
for each $v\in V$ there is $w\in W$ such that $\varphi (\int V\left(
v\right) )\#\int W\left( w\right) .$ In particular, if $v\in V^{1}$ then we
pick any $\sqsubset _{W}$-maximal $w\left( v\right) $ with this property$.$
Notice that $r_{W}\left( w\left( v\right) \right) =1$ if $W\left( w\left(
v\right) \right) $ is free and $r_{W}\left( w\left( v\right) \right) =0$
otherwise.

We shall construct a decreasing sequence of subsets $\left\{ T_{\alpha
}:\alpha \leq r\left( V\right) \right\} $ of $\max V$ such that for each $%
\alpha \leq r\left( V\right) $ and for every $\beta <\alpha ,$%
\begin{gather}
v\in V^{\alpha }\Longrightarrow T_{\alpha }\#\int V_{T_{\alpha }}\left(
v\right) ,  \label{S1} \\
\;x\in V^{\beta }\wedge \left\{ u\in V^{(\beta ,\alpha ]}:u\sqsubset
x\right\} =\varnothing \Longrightarrow x\in V_{T_{\alpha }},  \label{S2} \\
x\in V_{T_{\alpha }}^{\beta }\Longrightarrow V\left( x\right) \cap T_{\alpha
}=V\left( x\right) \cap T_{\beta }.  \label{S3}
\end{gather}

Let $T_{0}:=\max V$ and let%
\begin{equation*}
T_{1}=\bigcup \{V(v)\cap \varphi ^{-1}[\varphi (V(v))\cap W(w(v))]:v\in
V^{1}\}\cup (V\setminus V(V^{1})).
\end{equation*}

Clearly, $T_{1}$ fulfills (\ref{S1}-\ref{S3}) (here $\alpha =1,\beta =0)$.
Proceeding by induction on $1<\alpha \leq r\left( V\right) $, we assume that
$T_{\beta }$ has been defined for each $\beta <\alpha $ and the claim holds
for $\left\{ T_{\beta }:\beta <\alpha \right\} .$ Define%
\begin{equation*}
T_{<\alpha }:=\bigcap_{\beta <\alpha }T_{\beta }.
\end{equation*}

We claim (\footnote{%
{}In particular, by (\ref{S1}), $V_{S_{\alpha }}$ is a full subcascade of $V$
such that $V^{[\alpha ,r\left( V\right) ]}=V_{S_{\alpha }}^{[\alpha ,r\left(
V\right) ]}.$}) that for each $v\in V^{\alpha },$ the set $M\left( v\right) $
consisting of such $w\in W$ that there exists an infinite subset $B$ of $%
V_{T_{<\alpha }}^{+}\left( v\right) $ for which
\begin{equation}
\varphi \left( \max V_{T_{<\alpha }}\left[ B\right] \right) \subset W\left(
w\right) ,  \label{wl}
\end{equation}%
is not empty. By (\ref{S1}), $V_{T_{\alpha }}$ is a full subcascade of $V$
such that $V^{[\alpha ,r\left( V\right) ]}=V_{T_{\alpha }}^{[\alpha ,r\left(
V\right) ]},$ hence $\varnothing _{W}\in M\left( v\right) $ with $%
B=V_{T_{<\alpha }}^{+}\left( v\right) .$ As branches of a cascade are
finite, there exist $\sqsubset _{W}$-maximal elements of $M\left( v\right) .$
Let%
\begin{equation}
w\left( v\right) \in \max M\left( v\right)  \label{special}
\end{equation}%
be any of them, and let $B\left( v\right) $ be any infinite subset of $%
V_{T_{<\alpha }}^{+}\left( v\right) $ corresponding to $w\left( v\right) $
in (\ref{wl}).

Let%
\begin{equation}
T_{\alpha }:=\;\left( V\left[ B\left[ V^{\alpha }\right] \right] \cap
T_{<\alpha }\right) \cup \left( T_{<\alpha }\setminus V\left[ V^{\alpha }%
\right] \right) .  \label{alpha}
\end{equation}

Mind that the set $V\left[ V^{(r\left( V\right) ,\alpha ]}\setminus
V^{\alpha }\right] $ consists of those $x$ for which%
\begin{equation*}
\left\{ u\in V^{(r\left( x\right) ,\alpha ]}:u\sqsubset x\right\}
=\varnothing .
\end{equation*}

To see that (\ref{S1}) holds, let $v\in V^{\alpha }$. Then $B\left( v\right)
\subset B\left[ V^{\alpha }\right] ,$ hence $r\left( b\right) <\alpha $ for
each $b\in B\left( v\right) $, so that, by inductive (\ref{S1}), $T_{r\left(
b\right) }\#\int V_{T_{r\left( b\right) }}\left( b\right) ,$ and by
inductive (\ref{S2}), $V\left( b\right) \cap T_{<\alpha }=V\left( b\right)
\cap T_{r\left( b\right) }.$ As $B\left( v\right) $ is infinite, $V\left(
v\right) \cap T_{\alpha }=\bigcup_{b\in B\left( v\right) }\left( V\left(
b\right) \cap T_{<\alpha }\right) ,$ then $T_{\alpha }\#\int V_{T_{\alpha
}}\left( v\right) $ (\footnote{%
If $H\#\mathcal{A}_{n}$ for infinitely many $n$ then $H\#\int_{\left(
n\right) }\mathcal{A}_{n}.$}).

To see (\ref{S2}), let $x\in V^{\beta }$ such that $\left\{ u\in V^{(\beta
,\alpha ])}:u\sqsubset x\right\} =\varnothing .$ By the second part of (\ref%
{alpha}), $T_{<\alpha }\cap V\left[ V^{(\alpha ,r\left( V\right) ]}\setminus
V^{\alpha }\right] \subset T_{\alpha }$ so that
\begin{equation*}
T_{\alpha }\cap V\left( x\right) =T_{<\alpha }\cap V\left( x\right)
=T_{r\left( x\right) }\cap V\left( x\right) .
\end{equation*}%
To see (\ref{S3}), let $x\in V_{T_{\alpha }}^{\beta }.$ Then either

\begin{enumerate}
\item there exist $v\in V^{\alpha }$ and $y\in V^{+}\left( v\right) $ such
that $y\sqsubseteq x,$

\item or $x\notin V\left( V^{\alpha }\right) .$
\end{enumerate}

If (1) then, by inductive assumption, $V\left( x\right) \cap T_{r\left(
x\right) }=V\left( x\right) \cap T_{r\left( y\right) }$ and $V\left(
y\right) \cap T_{r\left( y\right) }=V\left( y\right) \cap T_{\alpha }$ by
the first part of (\ref{alpha}), so $V\left( x\right) \cap T_{r\left(
y\right) }=V\left( x\right) \cap T_{\alpha },$ because $V\left( x\right)
\subset V\left( y\right) .$ If (2), then by the second part of (\ref{alpha}),%
\begin{equation*}
T_{<\alpha }\cap \left( V\setminus V\left[ V^{\alpha }\right] \right)
=T_{\alpha }\cap \left( V\setminus V\left[ V^{\alpha }\right] \right) ,
\end{equation*}%
thus, by the inductive assumption (\ref{S3}), we are done.

Let $T:=T_{r\left( V\right) }.$ The proof will be complete if we show that
for each $v\in V_{T},$ the set%
\begin{equation}
\left\{ w\in W:\varphi \left( \int V_{T}\left( v\right) \right) \#\int
W\left( w\right) \right\}  \label{eos}
\end{equation}%
is either empty or a singleton. Indeed, $\varphi \left( \int V\right) \geq
\int W$ implies that $\varphi \left( T\right) \#\varphi \left( \int V\right)
\#\int W,$ thus $\varphi \left( \int V_{T}\right) \geq \int W,$ hence by
Corollary \ref{cor:her}, there exist $R,S$ such that for each $v\in V_{T\cap
R}$ there is a $w\in W_{S}$ such that $\varphi \left( \int V_{T\cap R}\left(
v\right) \right) \#\int W_{S}\left( w\right) .$
\end{proof}

\begin{lemmaA}
For each $v\in V_{T},$ the set (\ref{eos}) is either empty or a singleton.
\end{lemmaA}

\begin{proof}
We shall see that if $w$ is an element of (\ref{eos}), then $w=w\left(
v\right) ,$ as defined in (\ref{special}). Suppose that on the contrary,
there is a couple $\left( v,w\right) $ such that $\varphi \left( \int
V_{T}\left( v\right) \right) \#W\int \left( w\right) $ and $w\neq w\left(
v\right) .$ Let $\left( v_{0},w_{0}\right) $ be the least (with respect to
the lexicographic order of ranks) such a couple. Consider three cases.

\begin{enumerate}
\item $w\left( v_{0}\right) $ and $w_{0}$ are not $\sqsubset _{W}$%
-comparable,

\item $w_{0}\sqsubset _{W}w\left( v_{0}\right) ,$

\item $w\left( v_{0}\right) \sqsubset _{W}w_{0}.$
\end{enumerate}

Justification.

\begin{enumerate}
\item By (\ref{wl}) and (\ref{alpha}), $\varphi \left( \max V_{T}\left(
v_{0}\right) \right) \subset W\left( w\left( v_{0}\right) \right) $ and $%
W\left( w\left( v_{0}\right) \right) \cap W\left( w_{0}\right) =\varnothing
, $ hence $\varphi \left( \int V_{T}\left( v_{0}\right) \right) $ does not
mesh $\int W\left( w_{0}\right) .$

\item By (\ref{wl}) and (\ref{alpha}), $\varphi \left( \max V_{T}\left(
v_{0}\right) \right) \subset W\left( w\left( v_{0}\right) \right) $ and is $%
\int W\left( w\left( v_{0}\right) \right) $ is disjoint from $\int W\left(
w_{0}\right) .$

\item By Alternative (Theorem \ref{thm:alt}) applied to for $\varphi \left(
\int V_{T}\left( v_{0}\right) \right) $ and $\int W\left( w_{0}\right) ,$
one of the following holds:

\begin{enumerate}
\item $\varphi \left( \int V_{T}\left( v\right) \right) \#\int W\left(
w\right) $ for infinitely many $v\in V_{T}^{+}\left( v_{0}\right) $ and
infinitely many $w\in W^{+}\left( w_{0}\right) ,$ which contradicts the
maximality of $w\left( v_{0}\right) .$ Indeed, if $r\left( v\right) <r\left(
v_{0}\right) $ then $\varphi \left( \int V_{T}\left( v\right) \right) \#\int
W\left( w\right) $ implies $w=w\left( v\right) $, by minimality of $r\left(
v_{0}\right) $. Therefore if $v\in V_{T}^{+}\left( v_{0}\right) $ then%
\begin{equation*}
\varphi \left( \max V_{T}\left( v\right) \right) \subset \max W\left(
w\left( v\right) \right) ,
\end{equation*}%
where $w\left( v\right) \in W^{+}\left( w_{0}\right) ,$ hence $w_{0}\in
M\left( v_{0}\right) $, where $M$ was defined just before (\ref{wl}).

\item $\varphi \left( \int V_{T}\left( v\right) \right) \#\int W\left(
w_{0}\right) $ for infinitely many $v\in V_{T}^{+}\left( v_{0}\right) ,$
which is contradicted analogously to the preceding case.

\item $\varphi \left( \int V_{T}\left( v_{0}\right) \right) \#\int W\left(
w\right) $ for infinitely many $w\in W^{+}\left( w_{0}\right) ,$ in
contradiction with the minimality of $r\left( w_{0}\right) .$
\end{enumerate} \end{enumerate} \end{proof}

\section{Continuous maps between sequential cascades}

\begin{propositionA}
\label{prop:ext}Let $V,W$ be monotone sequential cascades such that $\varphi
\left( \int V\right) \geq \int W$ and for each $v\in V$ there exists a
unique $f\left( v\right) \in W$ such that $f\left( v\right) =\varphi \left(
v\right) $ if $v\in \max V$ and
\begin{equation}
\varphi \left( \int V\left( v\right) \right) \#\int W\left( f\left( v\right)
\right) .  \label{war}
\end{equation}%
Then there is a full subcascade $U$ of $V$ such that $\left. f\right\vert
_{U}$ is continuous.
\end{propositionA}

\begin{proof}
If $r\left( V\right) =1$ then $r\left( W\right) =1.$ Then it is enough to
set $U:=V_{\max W}.$ We use the induction with respect to the lexicographic
order (initializing for each $\left( \alpha ,1\right) .$ If $r\left(
W\right) =1,$ then $f\left( v\right) =\varnothing _{W}$ for each $v\in
V\setminus \max V$.

Let $\left( \alpha ,\beta \right) $ be the least couple, for which the claim
does not hold, and let $V$ and $W$ witness this and let $f:V\rightarrow W$
fulfill the assumptions. Note that $\beta >1.$ By Alternative theorem, one
of the following cases holds.

\begin{enumerate}
\item There exists an infinite subset $U_{0}$ of $V^{+}\left( \varnothing
_{V}\right) $ such that $f\left( U_{0}\right) $ is an infinite subset of $%
W^{+}\left( \varnothing _{W}\right) $. By inductive assumption, for each $%
v\in U_{0}$ there exists a full subcascade $U\left( v\right) $ of $V\left(
v\right) $ such that $\left. f\right\vert _{U\left( v\right) }$ is
continuous. Then, for $U:=\bigcup_{v\in U_{0}}U_{v}\cup \left\{ \varnothing
_{V}\right\} $ the restriction $\left. f\right\vert _{U}$ is continuous.

\item There exists an infinite subset $U_{0}$ of $V^{+}\left( \varnothing
_{V}\right) $ such that $f\left( U_{0}\right) =\varnothing _{W}.$ As above.

\item There exists an infinite subset $A$ of $W^{+}\left( \varnothing
_{W}\right) $ such that $f\left( \varnothing _{V}\right) =a$ for each $a\in
A,$ which yields a contradiction.
\end{enumerate}
\end{proof}

The following theorem considerably generalizes \cite[Theorem 3.3]{DSW},
where the function $\varphi $ was supposed to be injective. The proof of
\cite[Theorem 3.3]{DSW} was not correct.

\begin{theoremA}
\label{thm:ext} If $\varphi :\limfunc{ext}V\rightarrow \limfunc{ext}W$ is
continuous, then there exists a full subcascade $U$ of $V$ and a continuous
map $f:U\rightarrow W$ such that $\left. f\right\vert _{\limfunc{ext}%
U}=\left. \varphi \right\vert _{\limfunc{ext}U}$.
\end{theoremA}

\begin{proof}
If $\varphi :\limfunc{ext}V\rightarrow \limfunc{ext}W$ is continuous, then $%
\varphi \left( \int V\right) \geq \int W.$ Use Theorem \ref{thm:sharp} and
apply Proposition \ref{prop:ext} to $V_{S}$ and $W.$
\end{proof}

On the other hand, it is known that \cite[Proposition 3.2]{DSW}

\begin{propositionA}
\label{prop:rank}If $U$ and $W$ are monotone sequential cascades and $%
f:U\rightarrow W$ is a continuous map such that $f\left( \max U\right)
\subset \max W$ and $f\left( \varnothing _{U}\right) =\varnothing _{W}$, then%
\begin{equation*}
r_{U}\left( u\right) \geq r_{W}\left( f\left( u\right) \right)
\end{equation*}%
for each $u\in U.$
\end{propositionA}

As the rank is preserved by full subcascades, the two preceding fact imply
\cite[Lemma 8.4]{FrRaj}, that is, if $\alpha >\beta $ then there is no prime
map from a space of subsequential order $\beta $ to a space of subsequential
order $\alpha $ (see \cite[Corollary 3.4]{DSW} for terminology).

Simple examples like \cite[Example 3.1]{DSW}, show that in general a
continuous map $\varphi $ from Theorem \ref{thm:ext} has no continuous
extension to the whole cascade $V.$ By the way, Proposition \ref{prop:ult}
also implies the fact above.

Theorem \ref{thm:ext} has also applications in set theory and logic for some
classifications of ultrafilters on natural numbers (see \cite{StaroBeta},%
\cite{StaroP}).

\begin{corollaryA}
If $V,W$ are monotone sequential cascades for which $\int V=\int W,$ then $%
r\left( V\right) =r\left( W\right) .$
\end{corollaryA}

\begin{proof}
Suppose that $r\left( V\right) <r\left( W\right) $ and $\int V=\int W.$ If $%
X\in \int V$ then there is a full subcscade $U$ of $V_{X}$ such that the map
$\varphi (x):=x$ if $x\in X$ and $\varphi \left( \varnothing _{V}\right)
:=\varnothing _{W}$ coincides with a continuous map from $U$ to $W_{X}.$ By
Proposition \ref{prop:rank}, this leads a contradiction, because $r\left(
U\right) =r\left( V\right) <r\left( W_{X}\right) =r\left( W\right) .$
\end{proof}

This means that the rank of a contour is independent of a particular
monotone sequential cascade.

\begin{corollaryA}
\label{cor:rank}If $\mathcal{F},\mathcal{G}$ are monotone sequential
contours such that $r\left( \mathcal{G}\right) >r\left( \mathcal{F}\right) ,$
then $f\left( \mathcal{F}\right) \ngeqslant \mathcal{G}$ for each map $f$.
\end{corollaryA}

A subcascade $U$ of a cascade $V$ is called \emph{eventual} if $\varnothing
_{U}:=\varnothing _{V}$ and $U^{+}\left( v\right) $ is a cofinite subset of $%
V^{+}\left( v\right) $ for each $v\in U.$ Let $V,T$ be monotone sequential
cascades. We say that $T$ is a \emph{locally finite partition }of $V$ (in
symbols, $T\vartriangleright V)$ if there exist an eventual subcascade $U$
of $V$ and a surjection $f:T\rightarrow V$ with finite fibers such that%
\begin{gather}
f^{-1}\left( \varnothing _{V}\right) =\left\{ \varnothing _{T}\right\} ,
\label{s1} \\
f\left( T^{+}\left( f^{-}\left\{ v\right\} \right) \right) =V^{+}\left(
v\right) \text{ if }v\in U\setminus \max V,  \label{s2} \\
f\upharpoonright \max T\text{ is the identity.}  \label{s3}
\end{gather}%
This notion was introduced in \cite{StarosolskiCOU} (\footnote{%
It was formulated slightly differently.}).

\begin{propositionA}
If $T\vartriangleright V,$ then $\int T=\int V.$
\end{propositionA}

\begin{proof}
The claim is obviously true if $r\left( V\right) =1.$ Let $r\left( V\right)
=\alpha >1$ and suppose that the claim is true if $r\left( V\right) =\beta $
for each $\beta <\alpha .$ Then%
\begin{eqnarray*}
\int V &=&\limfunc{Li}\nolimits_{n\in V^{+}\left( \varnothing _{V}\right)
}\int V\left( v_{n}\right) \\
&=&\limfunc{Li}\nolimits_{n\in U^{+}\left( \varnothing _{U}\right) }\left(
\bigcap\nolimits_{f\left( p\right) =n}\int T\left( t_{p}\right) \right) \\
&=&\limfunc{Li}\nolimits_{p\in T^{+}\left( \varnothing _{T}\right) }\int
T\left( t_{p}\right) =\int T,
\end{eqnarray*}%
where in all the cases above, $\limfunc{Li}$ are taken over the
corresponding cofinite filters.
\end{proof}

\begin{theoremA}
If $V,W$ are monotone sequential cascades such that $\int V=\int W,$ then
there exists a monotone sequential cascade $T$ such that $T\vartriangleright
V$ and $T\vartriangleright W.$
\end{theoremA}

\begin{proof}
By Corollary \ref{cor:rank}, $r\left( V\right) =r\left( W\right) .$ We
induce on the common rank of $V$ and $W.$ If $r\left( V\right) =r\left(
W\right) =1,$ let $S:=\max V\cap \max W,$ and $T:=V_{S}=W_{S}.$ Then $\int
T=\int V=\int W,$ because $T$ is cofinite in $\max V$ and in $\max W.$

Let $r\left( V\right) =r\left( W\right) =\alpha >1$, and suppose that the
claim is true for each $\beta <\alpha .$ Represent $V^{+}\left( \varnothing
_{V}\right) =\left\{ v_{n}:n<\omega \right\} $ and $W^{+}\left( \varnothing
_{W}\right) =\left\{ w_{k}:k<\omega \right\} .$ By Theorem \ref{cor:trans},
the relation
\begin{equation*}
R:=\left\{ \left( n,k\right) :\int V\left( v_{n}\right) \#\int W\left(
w_{k}\right) \right\}
\end{equation*}%
is finite-to finite, and its domain and range are cofinite. Moreover,%
\begin{equation*}
\int V=\int W=\limfunc{Li}\nolimits_{\left( n,k\right) }\left( \int V\left(
v_{n}\right) \vee \int W\left( w_{k}\right) \right) ,
\end{equation*}%
where $\left( n,k\right) $ stands for the cofinite filter on $R.$ We start
defining a cascade $T,$ and maps $f_{V}:T\rightarrow V$ and $%
f_{W}:T\rightarrow W,$ by letting $T^{+}\left( \varnothing _{T}\right)
:=\left\{ t_{n,k}:\left( n,k\right) \in R\right\} $ and let%
\begin{equation*}
f_{V}:T^{+}\left( \varnothing _{T}\right) \rightarrow V^{+}\left(
\varnothing _{V}\right) \text{ and }f_{W}:T^{+}\left( \varnothing
_{T}\right) \rightarrow W^{+}\left( \varnothing _{W}\right)
\end{equation*}%
be defined by $f_{V}\left( t_{n,k}\right) :=v_{n}$ and $f_{W}\left(
t_{n,k}\right) :=w_{k}.$ By inductive assumption, for each $\left(
n,k\right) \in R,$ there exists a cascade $T_{n,k}$ such that maps $\varphi
_{n,k,V}:T_{n,k}\rightarrow V\left( v_{n}\right) $ and $\varphi
_{n,k,W}:T_{n,k}\rightarrow W\left( w_{k}\right) $ fulfilling the
conditions, \emph{mutatis mutandis,} (\ref{s1})-(\ref{s3}). We accomplish
defining a cascade $T$ for $t\sqsupset _{T}\left( n,k\right) $ by adjoining
to each $(n,k)$ the cascade $T_{n,k},$ and maps $f_{V}:T\rightarrow V$ and $%
f_{W}:T\rightarrow W,$ by gluing $\varphi _{n,k,V}$ to $f_{V}$ and $\varphi
_{n,k,W}$ to $f_{W}$ (see \cite{cascades} for details).
\end{proof}

A full subcascade $U$ of $V,$ to which the map $\varphi $ can be extended in
Theorem \ref{thm:ext}, is not unique in general. However, we cannot choose
\emph{a priori} an ultrafilter $\mathcal{Z}$ finer than $\int V$ so that $%
\max U\in \mathcal{Z}$.

\begin{exampleA}
Consider a monotone sequential cascade $V$ of rank $2.$ Let $V^{+}\left(
\varnothing _{V}\right) =\left\{ v_{n}:n<\omega \right\} $. Denote $%
X_{n}:=\max V\left( v_{n}\right) $ and, for each $0<n<\omega ,$ split $X_{n}$
into $n$ disjoint infinite sets $X_{n}=X_{n,1}\cup X_{n,2}\cup \ldots \cup
X_{n,n.}$ Define a sequential monotone cascade $W$ so that $\left\{
X_{n,k}:0<k\leq n,n<\omega \right\} $ coincide with $\left\{ \max W\left(
w\right) :r_{W}\left( w\right) =1\right\} $, and consider the identity $i:%
\limfunc{ext}V\rightarrow \limfunc{ext}W.$

Let $\mathbb{U}$ the set of $U,$ for which $i$ can be continuously extended
to a map from $U$ to $W.$ By Theorem \ref{thm:ext}, $\mathbb{U}\neq
\varnothing $ and $U\cap \left\{ v_{n}:n<\omega \right\} $ is infinite for
each $U\in \mathbb{U}$.

We shall construct an ultrafilter $\mathcal{Z}$ such that $\mathcal{Z}\geq
\int V=\int W$ and $\max U\notin \mathcal{Z}$ for each $U\in \mathbb{U}$.
Let $\mathcal{A}$ be the set of $A\subset \bigcup_{n<\omega }X_{n}$ such
that $\limfunc{card}\left\{ k:X_{n,k}\setminus A\neq \varnothing \right\}
\leq 1$ for each $n<\omega .$ Then $\mathcal{A}\#\int V,$ because if $%
A_{1},\ldots ,A_{p}\in \mathcal{A}$, then $A_{1}\cap \ldots \cap A_{p}\cap
X_{n}\neq \varnothing $ for $n>p.$ Let $\mathcal{Z}$ be an ultrafilter finer
than $\mathcal{A}\vee \int V.$ If $U\in \mathbb{U},\max U\in \mathcal{Z}$
and $f:U\rightarrow W$ is a continuous extension of $i:\limfunc{ext}%
V\rightarrow \limfunc{ext}W,$ then $\max U\in \mathcal{A}^{\#}$ and $\max
U\cap X_{n}\neq \varnothing $ for infinitely many $n.$ If $1\leq k\leq n$ is
such that $f\left( v_{n}\right) \in X_{n,k}$ then $X_{n}\supset
X_{n,k}\supset f^{-1}\left( \left( V^{+}\left( f\left( v_{n}\right) \right)
\right) _{0}\right) $ is not cofinite in $X_{n}$ for $n>1,$ hence $f$ is not
continuous at $v_{n}.$
\end{exampleA}

This example extends easily to

\begin{propositionA}
\label{prop:ult}For every monotone sequential cascade $V$ of rank greater
than $1,$ there exist a monotone sequential cascade $W$ and an ultrafilter $%
\mathcal{Z}\geq \int V=\int W$ such that if $U$ is a full subcascade of $V,$
for which there exists a continuous extension $f:U\rightarrow W$ of the
identity $i:\max V\rightarrow \max W,$ then $\max U\notin \mathcal{Z}$.
\end{propositionA}

Theorem \ref{thm:ext} can be interpreted in terms of subsequential spaces. A
topological space is called $\emph{prime}$ if at most one point is not
isolated. A map $\varphi $ from a prime space $X$ to a prime space is called
\emph{non-trivial} if it maps isolated points into isolated points, and the
non-isolated point of $X$ onto the non-isolated point of $Y.$ A subset $V$
of a prime space $X$ with a non-isolated point $o_{X}$ is called \emph{full}
if $o_{X}\in \limfunc{cl}\nolimits_{X}V\setminus V.$

\begin{propositionA}
Let $X$ and $Y$ be prime subsequential topological spaces and $\varphi
:X\rightarrow Y$ is a continuous non-trivial map. Then there exist
sequential spaces $\widetilde{V}$ and $\widetilde{Y}$, a full subset $V$ of $%
X$ and a continuous map $f:\widetilde{V}\rightarrow \widetilde{Y}$ such that
$V$ is a dense subset of $\widetilde{V},Y$ is a dense subset of $\widetilde{Y%
}$ and $\left. f\right\vert _{V}=\left. \varphi \right\vert _{V}.$
\end{propositionA}

\section{Appendix: Contour operation\label{app}}

We considered the contour $\limfunc{Li}\nolimits_{\left( n\right) }\mathcal{F%
}_{n}$ of a sequence $\left( \mathcal{F}_{n}\right) _{n}$ of filters. More
generally, if $\mathcal{B}\left( y\right) \subset 2^{X}$ for each $y\in Y,$
then the \emph{contour} of $\left\{ \mathcal{B}\left( y\right) :y\in
Y\right\} $ along $\mathcal{A}\subset 2^{Y}$ is defined by%
\begin{equation}
\mathcal{B}\left( \mathcal{A}\right) :=\limfunc{Li}\nolimits_{\mathcal{A}}%
\mathcal{B}\left( \cdot \right) =\bigcup\nolimits_{A\in \mathcal{A}%
}\bigcap\nolimits_{y\in A}\mathcal{B}\left( y\right) ,  \label{c}
\end{equation}%
that is, the \emph{set-theoretic lower limit} of subsets of $2^{X}$ along a
family $\mathcal{A}$. In this framework, the contour of a sequence is that
of $\left\{ \mathcal{F}_{n}:n<\omega \right\} $ along the \emph{cofinite
filter} $\left( n\right) ,$ on the natural numbers (\footnote{%
In some previous papers (\cite{cascades},\cite{DSW},\cite{multisequences}),
it was denoted by $\int_{\left( n\right) }\mathcal{F}_{n},$ that is,%
\begin{equation*}
\int_{\left( n\right) }\mathcal{F}_{n}:=\limfunc{Li}\nolimits_{\left(
n\right) }\mathcal{F}_{n}.
\end{equation*}%
}). This notion was introduced by G. H.\ Greco in \cite[p. 158]{ghg.ferrara}
for his theory of limitoids with important applications to the theory of $%
\Gamma $-functionals.

In the case of filters, the operation (\ref{c}) was introduced by H.\ J.\
Kowalsky in \cite{kowalskyA} (\footnote{%
Kowalsky inverted the order in the lattice of filters, so that $\bigcap $
and $\bigcup $ should be interchanged when translated to the usual order
(for which ultrafilters are maximal filters).}), where it was used for a
characterization of diagonality of convergence spaces \cite{kowalsky}. In
\cite{CookFischer}, C. H.\ Cook and H. R.\ Fisher referred to Kowalsky's
operation (\footnote{%
Defined as the set-theoretic upper limit (for the inverse order).}), under
the name of \emph{compression operator,} to characterize regularity of
convergence spaces. As mentioned in Section \ref{intr}, the contour can be
employed, in two complementary ways, for the characterization of both
diagonality and regularity \cite[Chapter V]{CFT}.

Z.\ Frol\'{\i}k introduced (\ref{c}) in \cite{Frolik} in the case of
ultrafilters, calling it a \emph{sum of ultrafilters}, and used it to prove
in ZFC his famous theorem on non-homogeneity of the Stone topology, showing
that there are $2^{\mathfrak{c}}$ types of ultrafilters.

In \cite[p. 53]{CFT} S. Dolecki and F.\ Mynard presented a more abstract
vision of the contour operation, namely, for each family $\mathcal{H}$ of
subsets of the set of families of subsets of a given set $X$ (\footnote{%
Specifically, in order to avoid an accumulation of exponents on five levels,
they denote $\S X:=2^{\left( 2^{X}\right) },$ so that $\mathcal{H}\in \S %
\left( \S X\right) $ and thus $\mathcal{H}^{\ast }\in \S X.$}), the contour $%
\mathcal{H}^{\ast }$ is defined by%
\begin{equation*}
\mathcal{H}^{\ast }:=\bigcup\nolimits_{H\in \mathcal{H}}\bigcap\nolimits_{%
\mathcal{B}\in H}\mathcal{B}.
\end{equation*}%
The definition (\ref{c}) can be easily seen to be a special case of the
notion above.

Considering the set-theoretic meaning of the notion, the term \emph{lower
limit} would be probably most precise to identify the object. However, we
opted for a shorter term \emph{contour}, intended to suggest a limiting
effect of a family $\mathcal{A}$ on the family of families $\left\{ \mathcal{%
B}\left( y\right) :y\in Y\right\} .$ The term \emph{sum} is most frequently
used by general topologists.


\begin{thebibliography}{99}

\bibitem{baum} {\footnotesize {J. E. Baumgartner}, {Ultrafilters on $\omega$}, {J. Symb. Logic},60(2): {624--639}, {1995}.}

\bibitem{CookFischer} {\footnotesize {C. Cook and H. Fischer}, {Uniform convergence structures}, {Math. Ann.},173: {290--306}, {1967}.}
	
\bibitem{multisequences} {\footnotesize {S. Dolecki}, {Multisequences}, {Quaestiones Mathematicae}, 29: 239--277, 2006,}
	
\bibitem{Tang} {\footnotesize {S. Dolecki and G. H. Greco}, {Towards historical roots of necessary conditions of optimality: Regula of Peano}, {Control and Cybernetics}, 36: {491--518}, 2007.}

\bibitem{cascades}  {\footnotesize {S. Dolecki and F. Mynard}, {Cascades and multifilters}, {Topology Appl.}, 104: 53--65, 2000.}
	
\bibitem{CFT} {\footnotesize  {S. Dolecki and F. Mynard}, {Convergence Foundations of Topology}, {World Scientific}, 2016.}
	

\bibitem{DSW} {\footnotesize {S. Dolecki and A. Starosolski and S. Watson},  {Comment. Math. Univ. Carolin.}, {44}(1): {165-181}, {2003}.}

\bibitem{DW.subseq}  {\footnotesize {S. Dolecki and S. Watson}, {Internal characterizations of subsequential topologies}, {preprint},2002.}
	

\bibitem{FrRaj}  {\footnotesize {S. Franklin and M. Rajagopalan}, {On subsequential spaces}, {Topology Appl.}, 35: 1--19, 1990.}
	

\bibitem{Frolik} {\footnotesize  {Z. Frol{\'\i}k}, {Sums of ultrafilters}, {Bull. Amer. Math. Soc.}, 73: 87--91, 1967.}
	

\bibitem{ghg.ferrara}  {\footnotesize {G. H. Greco}, {Limitoidi e reticoli completi}, {Ann. Univ. Ferrara}, 29: {153--164}, 1983.}
	

\bibitem{Jech} {\footnotesize {T. Jech}, {Set Theory}, {Academic Press}, 1978.}
	

\bibitem{kowalskyA} {\footnotesize {H. J. Kowalsky}, {{Beitr{\"a}ge zur topologischen Algebra}}, {Math. Nachr.}, 11: {143--185}, {1954}.}
	

\bibitem{kowalsky}    {\footnotesize {H. J. Kowalsky}, {Limesr{\"a}ume und {K}ompletierung}, {Math. Nachr.}, 12: {302--340}, {1954}.}
	
\bibitem{laflamme}   {\footnotesize {C. Laflamme}, {A few special ordinal ultrafilters}, {J. Symb. Logic}, 61(3): {920--927}, {1996}.}
	
\bibitem{MachStar} {\footnotesize {M. Machura and A. Starosolski}, {How high can Baumgartner's I-ultrafilters lie in the P-hierarchy?}, {Arch. Math. Logic}, {54}(5/60): {555-569}, {2015}.}
	
\bibitem{StaroBeta} {\footnotesize {A. Starosolski}, {P-hierarchy on $\beta \omega$}, {J. Symbolic Logic}, 73(4): {1202--1214}, {2008}.}
	
\bibitem{StarosolskiCOU} {\footnotesize {A. Starosolski}, {Cascades, order, and ultrafilters}, {Annals of Pure and Applied Logic}, {165}: {1626--1638},   {2014}.}
	
\bibitem{StaroP} {\footnotesize  {A. Starosolski}, {Ordinal ultrafilters versus {P}-hierarchy},  {Cent. Eur. J. Math.}, {12}: {84--96}, {2014}.}
	



\end{thebibliography}

\end{document}